\documentclass{amsart}
\usepackage{amsfonts}
\usepackage{amsmath,amssymb}
\usepackage{amsthm}
\usepackage{amscd}
\usepackage{graphics}
\usepackage{graphicx}
\theoremstyle{remark}{
\newtheorem{Def}{{\rm Definition}}

\newtheorem{Rem}{{\rm Remark}}

}

\newtheorem{Prop}{Proposition}
\newtheorem{Thm}{Theorem}

\newtheorem{MThm}{Main Theorem}

\begin{document}
\title[The images of special generic maps of several classes]{The images of special generic maps of several classes}
\author{Naoki Kitazawa}
\address{Institute of Mathematics for Industry, Kyushu University, 744 Motooka, Nishi-ku Fukuoka 819-0395, Japan\\
 TEL (Office): +81-92-802-4402 \\
 FAX (Office): +81-92-802-4405 \\
}
\email{n-kitazawa@imi.kyushu-u.ac.jp}
\urladdr{https://naokikitazawa.github.io/NaokiKitazawa.html}
\maketitle
\begin{abstract}
The class of {\it special generic} maps contains Morse functions with exactly two singular points, characterizing spheres topologically which are not $4$-dimensional and the $4$-dimensional unit sphere.
This class is for higher dimensional versions of such functions. Canonical projections of unit spheres are simplest examples and suitable manifolds diffeomorphic to ones represented as connected sums of products of spheres admit such maps.
They are known to restrict the topologies and the differentiable structures of the manifolds strongly.

Our paper focuses on images of special generic maps on closed manifolds. They are smoothly immersed compact manifolds whose dimensions are same as those of the targets. Some studies imply that they have much information on  (co)homology groups and cohomology rings for example. We present new construction and explicit examples of special generic maps from the images. Our paper is also essentially on construction and explicit algebraic topological and differential topological studies of immersed or embedded compact manifolds.
\end{abstract}
\section{Introduction.}
\label{sec:1}
{\it Special generic} maps are higher dimensiojnal versions of Morse functions with exactly two singular points. As Reeb's theorem shows, they characterize spheres topologically except $4$-dimensional cases and $4$-dimensional manifolds diffeomorphic to the unit sphere. These maps are attractive in the theory of algebraic topology and differential topology of manifolds. 

${\mathbb{R}}^k$ denotes the $k$-dimensional Euclidean space. This is regarded as a smooth manifold in the canonical way. This is also a Riemannian manifold endowed with the standard Euclidean metric. $||x|| \geq 0$ stands for the distance between $x$ and the origin $0$ under the previous metric.
$S^{k}\ {\rm (}D^k{\rm )}:=\{x \in {\mathbb{R}}^{k+1}\ {\rm (}resp. \  {\mathbb{R}}^{k}{\rm )} \mid ||x||=1\ {\rm (}resp. \ ||x|| \leq 1{\rm )}\}$ denotes the $k$-dimensional unit sphere (unit disk).

Hereafter, for a topoogical space $X$ homeomorphic to a CW complex, we can define the dimension $\dim X$ uniquely. (Topological) manifolds are such topological spaces. Smooth manifolds are well-known to be regarded as polyhedra in a natural way and regarded as PL manifolds. 
  
Throughout the present paper, a {\it singular point} $p \in X$ of a differentiable map $c:X \rightarrow Y$ is a point at which the rank of the differential ${dc}_p$ is smaller than $\min \{\dim X,\dim Y\}$. The {\it singular set} of the map is the set of all singular points of this. 

A {\it Special generic} map is a smooth ($C^{\infty}$) map from an $m$-dimensional manifold with no boundary into an $n$-dimensional manifold with no boundary at each {\it singular point} of which it is represented as
$(x_1,\cdots x_m) \mapsto (x_1,\cdots,x_n,{\Sigma}_{j=1}^{m-n} {x_{n+j}}^2)$ for suitable (local) coordinates. Canonical projections of unit spheres are also special generic. Their singular sets are equators and the restrictions to the singular sets are embeddings.

\begin{Prop}
\label{prop:1}
For a special generic map $c:X \rightarrow Y$, the singular set is a smooth closed submanifold of dimension $\dim Y-1$ of the manifold $X$ and has no boundary. Furthermore, the restriction of the map to the singular set is a smooth immersion.
\end{Prop}

{\it Fold} maps are defined as higher dimensional versions of Morse functions in a similar way. They also enjoy properties similar to ones of Proposition \ref{prop:1}. 

Special generic maps are fold maps. 
Fold maps are also fundamental and important tools in the theory of algebraic topology and differential topology of manifolds. 

\cite{golubitskyguillemin} is for systematic fundamental or advanced theory of singularity of differentiable maps and singularities for fold maps appaer. \cite{thom,whitney} are regarded as pionnering studies on fold maps, showing that a necessary and sufficient condition for a closed manifold whose dimension is at least $2$ to admit a fold map into ${\mathbb{R}}^2$ is that the Euler number is even for example.
See also \cite{eliashberg,eliashberg2} for the existence of fold maps. \cite{saeki} is one of pioneering studies on applications to algebraic topology and differential topology of manifolds. For example, in Theorem \ref{thm:4} and Remark \ref{rem:1} fold maps appear shortly. Except such scenes, we do not consider general fold maps in the present paper. 

We present Theorem \ref{thm:1}, showing explicitly that special generic maps give strong restrictions on the topologies and the differentiable structures of the manifolds in various cases. This makes special generic maps attractive.
 A {\it standard} sphere is a smooth manifold diffeomorphic to a unit sphere. A {\it homotopy sphere} is a smooth manifold homeomorphic to a unit sphere. An {\it exotic} sphere means a homotopy sphere which is not diffeomorphic to any standard sphere. 

\begin{Thm}[\cite{calabi, saeki2, saeki3, wrazidlo}]
\label{thm:1}
Any exotic sphere of dimension $m \geq 4$ admits no special generic maps into $\mathbb{R}^n$ for $n=m-3,m-2,m-1$. Homotopy spheres except $4$-dimensional exotic spheres admit special generic maps into ${\mathbb{R}}^2$.
Furthermore, $7$-dimensional oriented homotopy spheres of at least 14 types of all $28$ types admit no special generic maps into ${\mathbb{R}}^3$.
\end{Thm}

\cite{milnor} is well-known as a pioneering study on exotic spheres. The theory of \cite{eellskuiper} shows there exist exactly $28$ types of $7$-dimensional oriented homotopy spheres.     

More studies on special generic maps are presented in the third section. 

Our paper focuses on the images of special generic maps on closed manifolds. As is presented in the second section, they are smoothly immersed compact manifolds whose dimensions are same as those of the targets and they have much information on homology groups and cohomology rings in considerable cases. The present paper concerns essentially a constructive study of compact manifolds smoothly immersed or embedded in the Euclidean space of the same dimension to obtain special generic maps via Proposition \ref{prop:2} in the next section. Although such systematic studies on construction is very fundamental, even elementary studies and results seem to have not been known well until now due to the difficulty of construction of explicit manifolds and finding good applications to problems in geometry and general fields of mathematics. For example, Nishioka's proposition on restrictions on integral homology groups of compact and orientable manifolds, presented as Proposition \ref{prop:4} in the third section, is a very recent and explicit result. This yields a complete answer to the following problem in \cite{nishioka}: determine $5$-dimensional closed and simply-connected manifolds admitting special generic maps into ${\mathbb{R}}^n$ for $n=1,2,3,4$.  

We present Main results. Connected sums and boundary connected sums of manifolds are considered in the smooth category throughout the present paper. For fundamental or advanced notions, theory and notation on (co)homology and homotopy, see \cite{hatcher} for example. A {\it $Y/Z$-remove} to a compact and smooth manifold is defined as our new notion in the third section.
\begin{MThm}
	\label{mthm:1}
Let $G$ be an arbitrary finitely generated commutative group and $G_1$ and $G_2$ arbitrary free finitely generated commutative groups. Then there exists a $5$-dimensional compact and simply-connected manifold $X$ smoothly
embedded in ${\mathbb{R}}^5$ enjoying the following properties.
\begin{enumerate}
\item $X$ is represented as a boundary connected sum of manifolds diffeomorphic to either of the following two.
\begin{enumerate}
\item A manifold obtained by a $Y/Z$-remove to a $5$-dimensional standard sphere or a copy of the $5$-dimensional unit disk $D^5$ for a suitable $3$-dimensional closed, connected and orientable manifold $Y$ such that $H_1(Y;\mathbb{Z})$ is a finite
group which is not trivial and a one-point set $Z$.
\item $S^k \times D^{5-k}$ {\rm (}$2 \leq k \leq 4${\rm )}.
\end{enumerate}
\item $H_2(X;\mathbb{Z})$ is isomorphic to $G$. $H_j(X;\mathbb{Z})$ is isomorphic to $G_{j-2}$ for $j=3,4$.
\end{enumerate}
\end{MThm}
From Propositions \ref{prop:2} and \ref{prop:3} we have Main Theorem \ref{mthm:2}. $W_f$ here is introduced in the next section. This is, in short, the space of all connected components of the special generic map $f:M \rightarrow {\mathbb{R}}^n$ and an $n$-dimensional compact smooth manifold.
\begin{MThm}
	\label{mthm:2}
For any commutative group $A$, any integer $m>5$ and any manifold $X$ in Main Theorem \ref{mthm:1}, there exist an $m$-dimensional closed and simply-connected manifold $M$ and a special generic map $f:M \rightarrow {\mathbb{R}}^5$ such that $W_f$ is diffeomorphic to $X$, that the restriction to the singular set is an embedding and that $H_j(M;A) \cong H_j(X;A)$ for $0 \leq j \leq m-5$.
\end{MThm}
In the next section, we review fundamental algebraic topological and differential topological theory of special generic maps. We introduce fundamental construction, and (co)homological information of the manifolds admitting special generic maps for example. Key facts are that a special generic map $f$ is represented as the composition of a surjection onto a smooth manifold $W_f$ with a smooth immersion whose codimension is $0$ into the manifold of the target, that the manifold of the target of the first surjection has much (co)homological information of the manifold and that from a smooth immersion whose codimension is $0$ into a manifold with no boundary, we can construct a special generic on a suitable closed manifold. The third section is devoted to our main ingredient including Main Theorems. We first review known algebraic topological and differential topological studies of special generic maps including a fundamental result by Nishioka \cite{nishioka} or Proposition \ref{prop:4}. After that, we mainly present new construction of compact manifolds smoothly immersed or embedded in the Euclidean space of the same dimension. As presented in the next section, we have a special generic map on a suitable closed manifold from the immersion or embedding. As a result we can extend a result in \cite{kitazawa4}.
\section{Algebraic topological and differential topological properties of special generic maps.}
\label{sec:2}
The following proposition comes from a theorem characterizing manifolds admitting special generic maps in \cite{saeki2}. 
In the present paper, a diffeomorphism on a smooth manifold is assumed to be smooth and the {\it diffeomorphism group} of the manifold is the group consisting of all diffeomorphisms on the manifold, endowed with the so-called {\it Whitney $C^{\infty}$ topology}.  
A {\it smooth} bundle is a bundle whose fiber is a smooth manifold and whose structure group is the diffeomorphism group. A {\it linear} bundle is a smooth bundle whose fiber is a Euclidean space, a unit disk, or a unit sphere and whose structure group acts on the fiber as some groups of linear transformations.

\begin{Prop}[\cite{saeki2}]
\label{prop:2}
Let $m \geq n \geq 1$ be integers. Let ${\bar{f}}_N$ be a smooth immersion from an $n$-dimensional compact manifold $\bar{N}$ into an $n$-dimensional manifold $N$ with no boundary. In this situation, there exist some $m$-dimensional closed manifold $M$ and a special generic map $f:M \rightarrow N$ represented as the composition of a surjection $q_{f_N}$ onto $\bar{N}$ with ${\bar{f}}_N$ such that the following three properties are enjoyed hold.
\begin{enumerate}
\item The restriction of $q_{f_N}$ to the singular set of $f$ gives a diffeomorphism onto the boundary $\partial \bar{N}$.
\item The composition of the restriction of $q_{f_N}$ to the preimage of some small collar neighborhood $N(\partial \bar{N})$ with the canonical projection to the boundary gives a trivial linear bundle whose fiber is diffeomorphic to the {\rm (}$m-n+1${\rm )}-dimensional unit disk $D^{m-n+1}$.
\item On the preimage of the complementary set of the interior of the collar neighborhood $N(\partial \bar{N})$, $q_{f_N}$ gives a trivial smooth bundle whose fiber is an {\rm (}$m-n${\rm )}-dimensional standard sphere.
\end{enumerate}
If $N$ is orientable, then we can obtain $M$ as an orientable manifold. If $N$ is connected, then we can obtain $M$ as a connected manifold. If $N$ satisfies both, then we can construct $M$ satisfying both.
\end{Prop}
In general, special generic maps have structures in Proposition \ref{prop:2} where the two bundles may not be trivial (see \cite{saeki2}).
Hereafter, let $q_{f_N}$, $\bar{N}$ and ${\bar{f}}_N$ denote by $q_f$, $W_f$ and $\bar{f}$, respectively. This is based on well-known notation on the so-called {\it Stein factorizations} and the so-called {\it Reeb spaces} of maps between topological spaces. The notion of the {\it Reeb space} of a (continuous) map (between topological spaces) appear shortly again in Remark \ref{rem:1} for example.

We can know simple examples explaining well about special generic maps in \cite{saeki2}
 and also in \cite{kitazawa3, kitazawa4} for example.
\begin{Prop}[E. g. \cite{kitazawa2, kitazawa3, saeki, saeki2, saekisuzuoka}]
\label{prop:3}
Let $A$ be a commutative group. In the situation of Proposition \ref{prop:2}, $q_f:M \rightarrow W_f:=\bar{N}$ induces isomorphisms ${q_f}_{\ast}:H_j(M;A) \cong H_j(W_f;A)$, ${q_f}_{\ast}:{\pi}_j(M) \cong {\pi}_j(W_f)$, and ${q_f}^{\ast}:H^{j}(W_f;A) \cong H^{j}(M;A)$ for $0 \leq j \leq m-n$ where as an additional assumption $M$ is assumed to be connected in discussing homotopy groups for example.
\end{Prop}

This is in our several scenes important.
This holds for general special generic maps and more general smooth maps of certain classes. See the articles referred to here. Remark \ref{rem:1} refers to this shortly for example.

\section{Compact manifolds smoothly immersed or embedded in the Euclidean space of the same dimension and Main Theorems.}
\label{sec:3}
\subsection{Several algebraic topological and differential topological studies of special generic maps including ones on $W_f$ for a special generic map $f$ in Proposition \ref{prop:2}.}
The following was shown in \cite{nishioka} and the proof works in the topology category.

\begin{Prop}[\cite{nishioka}]
\label{prop:4}
Let $A$ be a principal ideal domain. Let $P$ be a compact, connected and orientable manifold whose dimension is greater than or equal to $2$ satisfying $H_1(P;A)=\{0\}$. Then, $H_{j}(P;A)$ is free for $j=\dim P-2,\dim P-1$. Furthermore, if $A=\mathbb{Z}$, then without assuming that $P$ is orientable, $P$ is shown to be orientable.
\end{Prop}
Note that we can easily know this for $\dim P=2,3$. Nishioka applied this for $\dim P=4$ and $\partial P \neq \emptyset$ to determine dimensions of Eulidean spaces into which $5$-dimensional closed and simply-connected manifolds admit special generic maps.
$5$-dimensional closed and simply-connected manifolds are studied in
\cite{barden} for example. They are classified completely in the topology, the PL, and
the smooth categories there.
\begin{Thm}[\cite{nishioka}]
\label{thm:2}
A $5$-dimensional closed and simply-connected admits a special generic map into ${\mathbb{R}}^n$ for $n=3,4$ if and only if it is homeomorphic to {\rm (}and as a result diffeomorphic to{\rm )} a $5$-dimensional standard sphere or represented as a connected sum of the total spaces of {\rm (}linear{\rm )} bundles over the $2$-dimensional unit sphere $S^2$ whose fiber is the $3$-dimensional unit sphere $S^3$.
\end{Thm}
We introduce several known results on manifolds admitting special generic maps of suitable classes.
\begin{Thm}[\cite{saeki2}.]
\label{thm:3}
A closed and connected manifold $M$ {\rm (}whose dimension is at least $2${\rm )} admits a special generic map into ${\mathbb{R}}^2$ if and only if either of the following two hold. 
\begin{enumerate}
	\item $M$ is a homotopy sphere which is not an $4$-dimensional exotic sphere.
	\item $M$ is represented as a connected sum of the total spaces of smooth bundles over the circle whose fibers are homotopy spheres which are not $4$-dimensional exotic spheres.
\end{enumerate}
\end{Thm}
Furthermore, closed manifolds admitting special generic maps into ${\mathbb{R}}^3$ are classified under constraints on the fundamental groups.
\begin{Thm}[\cite{saekisakuma}]
\label{thm:4}
A $4$-dimensional closed and connected manifold $M$ whose fundamental groups is free admits a special generic map into ${\mathbb{R}}^3$ if and only if 
 if and only if either of the following two hold. 
\begin{enumerate}
	\item $M$ is a $4$-dimensional standard sphere.
	\item $M$ is represented as a connected sum of the total spaces of smooth {\rm (}linear{\rm )} bundles over the circle whose fibers are the $3$-dimensional unit sphere $S^3$ or ones over $S^2$ whose fibers are the $2$-dimensional unit sphere $S^2$.
	
\end{enumerate}
\end{Thm}
\begin{Thm}[E.g. \cite{saekisakuma2}]
\label{rhm:5}
	For some manifold $M$ in the situation of Theorem \ref{thm:4},
there exists a closed and connected smooth manifold $M^{\prime}$ enjoying the following three properties.
\begin{enumerate}
\item $M^{\prime}$ is homeomorphic to $M$.
\item $M^{\prime}$ admits a fold map into ${\mathbb{R}}^3$.
\item $M^{\prime}$ admits no special generic maps into ${\mathbb{R}}^3$.
\end{enumerate}
\end{Thm}
\subsection{A kind of surgery operations to construct compact manifolds smoothly immersed or embedded into the same dimensional Euclidean spaces and main results.}
Let $X$ be a topological space regarded as a CW complex. We can generalize more. However we concentrate on such a case.

We explain about the (integral) homology group of $X \times S^k$ for $k>0$.
For $H_i(X \times S^k;\mathbb{Z})$ and an arbitrary integer $i$, we can define $H_{b,i,k}(X) \subset H_i(X \times S^k;\mathbb{Z})$ as the set of all integral homology classes represented by cycles of the form ${i_p} {\ast}(c)$ where $c$ is a cycle of $X$ and $i_p$ is an inclusion of the form $i_p:X \rightarrow X \times S^k$ satisfying $i_p(x)=(x,p)$ ($p \in S^k$). This is a subgroup and well-defined: this does not depend on $p \in S^k$. Let $H_{f,i,k}(X) \subset H_i(X \times S^k;\mathbb{Z})$ be the set of all integral homology classes in the image of a kind of variants of so-called {\it prism operators} from $H_{b,i-k,k}(X)$ to $H_i(X \times S^k;\mathbb{Z})$. This is defined respecting the structure of the product. This homomorphism is a monomorphism. As a result, $H_i(X \times S^k;\mathbb{Z})$ is the internal direct sum of $H_{b,i,k}(X)$ and $H_{f,i,k}(X)$.
In short, this comes from so-called
K\"unneth theorem. See \cite{hatcher} for example.

Let $X$ be a smooth, compact and connected manifold and $Y$ a connected closed submanifold of $X$ with no boundary embedded smoothly in ${\rm Int}\ X$ so that its normal bundle is trivial.
Let $Z$ be the empty set or a connected closed submanifold of $Y$ with no boundary satisfying $\dim Z<\dim Y$ and embedded smoothly in $Y$.
Let $N(Y)$ be a small closed tubular neighborhood of $Y$ in $X$ and $N(Z)$ a small closed tubular neighborhood of $Z \subset Y$ (in $Y$ if $Z \neq \emptyset$). $N(Y)$ is regarded as the total space of a trivial linear bundle whose fiber is diffeomorphic to $D^{\dim X-\dim Y}$, regarded as a subbundle of a (trivial) normal bundle of $Y$, which is also a linear bundle whose fiber is ${\mathbb{R}}^{\dim X-\dim Y}$. We regard $N(Y)$ as such a trivial linear bundle over $Y$ and let $N(Y,Z) \subset N(Y)$ be the restriction of this bundle to the new base space
$N(Z)$. $E(Y,Z)$ denotes the complementary set of the interior of $N(Y,Z)$ in $N(Y)$. This is regarded as the restriction of the bundle $N(Y)$ over $Y$ to the complementary set $CE(Z)$ of the interior of $N(Z)$ in $Y$ as the new base space. We can remove the interior of $E(Y,Z)$ from $X$ and obtain a ($\dim X$)-dimensional compact smooth manifold with no corner by smoothing in a well-known natural way.
\begin{Def}
We call the procedure of removing the interior of $E(Y,Z)$ from $X$ and smoothing to obtain a compact smooth manifold with no corner a {\it $Y/Z$-remove} to $X$.
\end{Def}

Note that for example, if $\dim X-\dim Y \geq 2$, then the resulting manifold is connected. Note also that by smoothing corners of an arbitrary general smooth manifold, we always have mutually diffeomorphic manifolds. Hereafter
let $X\overline{(Y,Z)}$ denote the resulting manifold.

We investigate (co)homology groups and homotopy groups (fundamental groups) of thess spaces (manifolds) and present our Main Theorems with proofs. \cite{hatcher} also presents strong methods for these studies systematically. For example, homology exact sequences such as ones for pairs of topological spaces and Mayer-Vietoris sequences are important tools.   

Let $NE(Y,Z)$ denote $E(Y,Z) \bigcap N(Y,Z)$. This is regarded as the restriction of the two bundles $E(Y,Z)$ and $N(Y,Z)$ to the new base space $\partial N(Z)$.
Let $SE(Y,Z)$ denote the subbundle of $E(Y,Z)$, which is a bundle over the complementary set $CE(Z)$ of the interior of $N(Z)$ in $Y$ obtained by restricting the fiber to the boundary of the ($\dim X-\dim Y$)-dimensional disk of the fiber of the given bundle. Let $SNE(Y,Z)$ denote the subbundle of $NE(Y,Z)$ obtained in a similar way. We have a Mayer-Vietoris exact sequence
$$\rightarrow H_j(SNE(Y,Z);\mathbb{Z}) \rightarrow H_j(NE(Y,Z);\mathbb{Z}) \oplus H_j(SE(Y,Z);\mathbb{Z}) \rightarrow H_j(\partial E(Y,Z);\mathbb{Z}) \rightarrow$$
and an equivalent exact sequence
$$\rightarrow H_{b,j,\dim X-\dim Y-1}(\partial N(Z)) \oplus H_{f,j,\dim X-\dim Y-1}(\partial N(Z))$$
$$\rightarrow H_j(\partial N(Z);\mathbb{Z}) \oplus H_{b,j,\dim X-\dim Y-1}(CE(Z)) \oplus H_{f,j,\dim X-\dim Y-1}(CE(Z))$$
$\rightarrow H_j(\partial E(Y,Z);\mathbb{Z}) \rightarrow$\\
\ \\
where we have this equivalent exact sequence by the following arguments.
\begin{itemize}
	\item By the definition, $NE(Y,Z)$ is the total space of a trivial linear bundle over $\partial N(Z)$ whose fiber is diffeomorphic to the ($\dim X-\dim Y$)-dimensional unit disk $D^{\dim X-\dim Y}$ and $SNE(Y,Z)$ is the subbundle obtained by restricting the fiber to the boundary.
	\item The previous bundle $NE(Y,Z)$ is regarded as the restrictions of the trivial linear bundle $E(Y,Z)$ over $CE(Z)$ and the trivial linear bundle over $N(Z)$, where these two bundles are the restrictions of the trivial linear bundle $N(Y)$ over $Y$ whose fiber is diffeomorphic to the ($\dim X-\dim Y$)-dimensional unit disk $D^{\dim X-\dim Y}$. 
	This, the previous argument and the definitions of some groups such as $H_{b,j,\dim X-\dim Y-1}(\partial N(Z))$ and $H_{f,j,\dim X-\dim Y-1}(\partial N(Z))$, introduced before, give an isomorphism between $H_j(SNE(Y,Z);\mathbb{Z})$ and $H_{b,j,\dim X-\dim Y-1}(\partial N(Z)) \oplus H_{f,j,\dim X-\dim Y-1}(\partial N(Z))$.
	\item The total space of the bundle $NE(Y,Z)$ collapses to $\partial N(Z)$ where $\partial N(Z)$ here is regarded as the image of a section of the trivial linear bundle. This gives an isomorphism between 
	$H_j(NE(Y,Z);\mathbb{Z})$ and $H_j(\partial N(Z);\mathbb{Z})$.
\item $SE(Y,Z)$ is the subbundle of the bundle $E(Y,Z)$ over $CE(Z)$ obtained by restricting the fiber to the boundary. This gives an isomorphism between $H_j(SE(Y,Z);\mathbb{Z})$ and $H_{b,j,\dim X-\dim Y-1}(CE(Z)) \oplus H_{f,j,\dim X-\dim Y-1}(CE(Z))$.

\end{itemize}

 We can represent the homomorphism from $$H_{b,j,\dim X-\dim Y-1}(\partial N(Z)) \oplus H_{f,j,\dim X-\dim Y-1}(\partial N(Z))$$ by
 the (direct) sum of the two homomorphisms $$\{i_{j^{\prime},1} \oplus i_{j^{\prime}.2} \oplus i_{j^{\prime},3}\}_{j^{\prime}=1}^2$$ into the direct sum $$H_j(\partial N(Z);\mathbb{Z}) \oplus H_{b,j,\dim X-\dim Y-1}(CE(Z)) \oplus H_{f,j,\dim X-\dim Y-1}(CE(Z))$$
 where $i_{j^{\prime},{j^{\prime \prime}}}$ is a homomorphism from the $j^{\prime}$-th summand into the $j^{\prime \prime}$-th summand.  $i_{j^{\prime},{j^{\prime \prime}}}$ is also defined as a suitable homomorphism from a suitable subgroup of the original group of the domain into a suitable subgroup of the group of the target.
We can represent the homomorphism from $$H_j(\partial N(Z);\mathbb{Z}) \oplus H_{b,j,\dim X-\dim Y-1}(CE(Z)) \oplus H_{f,j,\dim X-\dim Y-1}(CE(Z))$$ by the sum of the three homomorphisms in $\{{i^{\prime}}_{j^{\prime}}\}_{j^{\prime}=1}^3$ into $H_j(\partial E(Y,Z);\mathbb{Z})$ where the $j^{\prime}$-th homomorphism is a homomorphism from the $j^{\prime}$-th summand. ${i^{\prime}}_{j^{\prime}}$ is also defined as a suitable homomorphism from a suitable subgroup of the original group of the domain into the group of the target.

Related to this, we prove some propositions one by one. By applying some of the obtained statements and arguments, we prove Main Theorem \ref{mthm:1}.

First, we investigate the homomorphisms and $H_j(\partial E(Y,Z);\mathbb{Z})$ for $j>0$ under the conditions that $\dim X-\dim Y \geq 2$, that $\dim Y \geq 2$ and that $Z$ is a one-point set.

\begin{Prop}
	In the present situation, let $Z$ be a one-point set and assume also that $\dim X-\dim Y \geq 2$ and $\dim Y \geq 2$. We have the following facts.
	\label{prop:5}
	\begin{enumerate}
		\item \label{prop:5.0}
		$CE(Z)$ is connected.
		\item
		\label{prop:5.1}
		Hereafter $i_{j^{\prime},l,j}$ and ${i^{\prime}}_{j^{\prime},j}$ denote homomorphisms from the groups whose degrees are $j$ into ones whose degrees are $j$.   
		\begin{enumerate}
					\item \label{prop:5.1.1} $i_{1,1,j}$ an isomorphism for any $j$. $i_{1,2,j}$ is the zero homomorphism for $j \neq 0$. For $j=0$, $i_{1,2,j}$ is an isomorphism between groups isomorphic to $\mathbb{Z}$. $i_{1,3,j}$ is the zero homomorphism for any $j$.
		\item \label{prop:5.1.2}
		$i_{2,1,j}$ is the zero homomorphism for any $j$. $i_{2,2,j}$ is the zero homomorphism for any $j$. $i_{2,3,j}$ is the zero homomorphism from the trivial group for $j<\dim X-\dim Y-1$ and $\dim X-\dim Y-1<j < \dim X-2$, an isomorphism between groups isomorphic to $\mathbb{Z}$ for $j=\dim X-\dim Y-1$, the zero homomorphism from a group isomorphic to $\mathbb{Z}$ for $j=\dim X-2$.
		\item \label{prop:5.1.3} ${i^{\prime}}_{1,j}$ is the zero homomorphism for any $j$.
		The sum of the three homomorphisms in
		$\{{i^{\prime}}_{j^{\prime},j}\}_{j^{\prime}=1}^3$ is an epimorphism for $1 \leq j \leq \dim X-2$. 
		\end{enumerate}
	\item 
	\label{prop:5.2}
	 $H_j(\partial E(Y,Z);\mathbb{Z})$ is isomorphic to the direct sum of $H_{j}(CE(Z);\mathbb{Z})$ and $H_{j-(\dim X-\dim Y-1)}(CE(Z);\mathbb{Z})$ for $\dim X-\dim Y-1<j \leq \dim X-2$. 
	\item
	\label{prop:5.3} $H_j(\partial E(Y,Z);\mathbb{Z})$ is isomorphic to the direct sum of $H_{j}(CE(Z);\mathbb{Z})$ and the trivial group for $1 \leq j \leq \dim X-\dim Y-1<\dim X-2$. 

	\end{enumerate}
\end{Prop}
\begin{proof}
	First, $CE(Z)$ is connected by our definitions and we have (\ref{prop:5.0}).

We investigate the direct sum	
$i_{1,1,j} \oplus i_{1.2,j} \oplus i_{1,3,j}$. The first summand is an isomorphism for every $j$ since it is regarded as an identity map under a suitable identification by our definition of the group $H_{b,j,\dim X-\dim Y-1}(\partial N(Z))$. The second summand is the zero homomorphism for any $j \neq 0$ by the fact that $\partial N(Z)$ is a sphere of dimension $\dim Y-1>1$ bounding $CE(Z)$ and the definitions of the groups $H_{b,j,\dim X-\dim Y-1}(\partial N(Z))$ and $H_{b,j,\dim X-\dim Y-1}(CE(Z))$. 
This is also an isomorphism between groups isomorphic to $\mathbb{Z}$ for $j=0$ by the same reason with (\ref{prop:5.0}). 
The third summand is the zero homomorphism by the definitions of $H_{b,j,\dim X-\dim Y-1}(\partial N(Z))$ and $H_{f,j,\dim X-\dim Y-1}(CE(Z))$. We have (\ref{prop:5.1.1}).

$i_{2,1,j} \oplus i_{2,2.j} \oplus i_{2,3,j}$ is the direct sum of a zero homomorphism, another zero homomorphism and a homomorphism. We know this for the first two homomorphisms by our definitions of $H_{f,j,\dim X-\dim Y-1}(\partial N(Z))$ and $H_{b,j,\dim X-\dim Y-1}(CE(Z))$. By our definitions of $H_{f,j,\dim X-\dim Y-1}(\partial N(Z))$ and $H_{f,j,\dim X-\dim Y-1}(CE(Z))$, the assumption that $\partial N(X)$ is a sphere of dimension $\dim Y-1>1$ and (\ref{prop:5.0}) for example, the third homomorphism $i_{2,3,j}$ is the zero homomorphism from the trivial group for $j<\dim X-\dim Y-1$
and $\dim X-\dim Y-1<j<\dim X-2$, an isomorphism between groups isomorphic to $\mathbb{Z}$ for $j=\dim X-\dim Y-1$, and the zero homomorphism from a group isomorphic to $\mathbb{Z}$ for $j=\dim X-2$. This is also due to an argument similar to one for $i_{1,2,j}$. We have (\ref{prop:5.1.2}).

By the fact that $i_{1,1,j}$ is an isomorphism, ${i^{\prime}}_{1,j}$ is the zero homomorphism for every $j$.
 For example, the facts that $i_{1,1,j}$ is an isomorphism for every $j$, that $i_{1,3,j}$ and $i_{2,1,j}$ are the zero homomorphisms for every $j$ and that for $1 \leq j <\dim X-2$, $i_{2,3,j}$ is a monomorphism, and the exact sequence imply that the sum of the three homomorphisms in
$\{{i^{\prime}}_{j^{\prime},j}\}_{j^{\prime}=1}^3$ is an epimorphism for $1 \leq j \leq \dim X-2$. We have (\ref{prop:5.1.3}). This completes the proof of (\ref{prop:5.1}).

 (\ref{prop:5.1}) and this exact sequence also imply that for $\dim X-\dim Y-1<j \leq \dim X-2$, $H_j(\partial E(Y,Z);\mathbb{Z})$ is isomorphic to the direct sum of $H_{j}(CE(Z)
 ;\mathbb{Z})$, identified with the image of $i^{\prime}_{2,j}$, and $H_{j-(\dim X-\dim Y-1)}(CE(Z);\mathbb{Z})$, identified with the image of $i^{\prime}_{3,j}$. We have (\ref{prop:5.2}).

 (\ref{prop:5.1}) and this exact sequence also imply that for $1 \leq j \leq \dim X-\dim Y-1 <\dim X-2$, $H_j(\partial E(Y,Z);\mathbb{Z})$ is isomorphic to the direct sum of $H_{j}(CE(Z);\mathbb{Z})$, identified with the image of $i^{\prime}_{2,j}$, and another group $G_j$. We can see that $G_j$ is isomorphic to the quotient group $H_{f,j,\dim X-\dim Y-1}(CE(Z))/i_{2,3,j}(H_{f,j,\dim X-\dim Y-1}(\partial N(Z)))$ of $H_{f,j,\dim X-\dim Y-1}(CE(Z))$. This is always the trivial group for $1 \leq j \leq \dim X-\dim Y-1 <\dim X-2$ by topological properties of $i_{2,3,j}$. We have (\ref{prop:5.3}). 

This completes the proof.
\end{proof}
\begin{Prop}
\label{prop:6}
Let $X$ be a copy of a unit disk. Let $Y$ be a closed submanifold of $X$ with no boundary satisfying $\dim Y>0$ and embedded smoothly in ${\rm Int}\ X$ so that its normal bundle is trivial and that the dimension of the normal bundle is greater than $1${\rm :} it is assumed that $\dim X-\dim Y \geq 2$ is satisfied.
Let $Z$ be a one-point set of $Y$.

 In this situation, $H_j(X\overline{(Y,Z)};\mathbb{Z})$ is isomorphic to $H_{j-(\dim X-\dim Y-1)}(CE(Z);\mathbb{Z})$ for $\dim X-\dim Y-1<j<\dim X-1$, the trivial group for $1 \leq j \leq \dim X-\dim Y-1$ and $\mathbb{Z}$ for $j=\dim X-1$. Furthermore, $X\overline{(Y,Z)}$ is simply-connected.
\end{Prop}
\begin{proof}
We have a Mayer-Vietoris exact sequence \\
 \\
$\rightarrow H_j(\partial E(Y,Z);\mathbb{Z}) \rightarrow H_j(E(Y,Z);\mathbb{Z}) \oplus H_j(X\overline{(Y,Z)};\mathbb{Z}) \rightarrow H_j(X;\mathbb{Z}) \rightarrow$ \\
 \\
and since $X$ is a disk for $j \geq 1$ the homomorphism from $H_j(\partial E(Y,Z);\mathbb{Z})$ into $H_j(E(Y,Z);\mathbb{Z}) \oplus H_j(X\overline{(Y,Z)};\mathbb{Z})$ is an isomorphism.

Here we review Proposition \ref{prop:5}, presented just before. $H_j(\partial E(Y,Z);\mathbb{Z})$ is shown to be isomorphic to the direct sum of $H_{j}(CE(Z);\mathbb{Z})$ and $H_{j-(\dim X-\dim Y-1)}(CE(Z);\mathbb{Z})$ for $\dim X-\dim Y-1 < j \leq \dim X-2$ in Proposition \ref{prop:5} (\ref{prop:5.2}).
It is also shown to be isomorphic to the direct sum of $H_{j}(CE(Z);\mathbb{Z})$ and the trivial group for $1 \leq j \leq \dim X-\dim Y-1$ in Proposition \ref{prop:5} (\ref{prop:5.3}).

The summand
$H_{j-(\dim X-\dim Y-1)}(CE(Z);\mathbb{Z})$ in the case $\dim X-\dim Y-1 < j \leq \dim X-2$ is regarded as the image of ${i^{\prime}}_3:H_{f,j,\dim X-\dim Y-1}(CE(Z)) \rightarrow H_j(\partial E(Y,Z);\mathbb{Z})$. It is mapped onto
the summand $H_j(E(Y,Z);\mathbb{Z})$ of $H_j(E(Y,Z);\mathbb{Z}) \oplus H_j(X\overline{(Y,Z)};\mathbb{Z})$ by the zero homomorphism and into $H_j(X\overline{(Y,Z)};\mathbb{Z})$ as a monomorphism in this Mayer-Vietoris sequence. This is due to the fact that the homomorphism from $H_j(\partial E(Y,Z);\mathbb{Z})$ into $H_j(E(Y,Z);\mathbb{Z}) \oplus H_j(X\overline{(Y,Z)};\mathbb{Z})$ is an isomorphism for $j \geq 1$ and our definition of $H_{f,j,\dim X-\dim Y-1}(CE(Z);\mathbb{Z})$ with the structures of our trivial bundles.

The first summands of both cases are $H_{j}(CE(Z);\mathbb{Z})$.
Each of them is regarded as the image of ${i^{\prime}}_{2,j}:H_{b,j,\dim X-\dim Y-1}(CE(Z)) \rightarrow H_j(\partial E(Y,Z);\mathbb{Z})$ before.
 In the Mayer-Vietoris sequence here, it is mapped onto the summand $H_j(E(Y,Z);\mathbb{Z})$ of $H_j(E(Y,Z);\mathbb{Z}) \oplus H_j(X\overline{(Y,Z)};\mathbb{Z})$ isomorphically due to our definition of $H_{b,j,\dim X-\dim Y-1}(CE(Z);\mathbb{Z})$ with the structures of our trivial bundles. It is also mapped into the summand $H_j(X\overline{(Y,Z)};\mathbb{Z})$ in this Mayer-Vietoris sequence by some homomorphism. 

In our Mayer-Vietoris sequence here, the homomorphism from $H_j(\partial E(Y,Z);\mathbb{Z})$ into $H_j(E(Y,Z);\mathbb{Z}) \oplus H_j(X\overline{(Y,Z)};\mathbb{Z})$ is an isomorphism for $j \geq 1$. From our arguments here, we can see that
$H_{j}(X\overline{(Y,Z)};\mathbb{Z})$ is isomorphic to the summand $H_{j-(\dim X-\dim Y-1)}(CE(Z);\mathbb{Z})$ for $\dim X-\dim Y-1<j<\dim X-1$ and the trivial group for $1 \leq j \leq \dim X-\dim Y-1 \geq 1$. 

$\partial X\overline{(Y,Z)}$ consists of exactly two connected components and
 $H_{1}(X\overline{(Y,Z)};\mathbb{Z})$ is the trivial group. By a kind of exercises on the homology exact sequence for $(X\overline{(Y,Z)},\partial X\overline{(Y,Z)})$,$H_{1}(X\overline{(Y,Z)},\partial E(Y,Z);\mathbb{Z})$ is free and of rank $1$ and $H^1(X\overline{(Y,Z)},\partial E(Y,Z);\mathbb{Z})$ is also free and of rank $1$. By Poincar\'e duality theorem for $X\overline{(Y,Z)}$, $H_{\dim X-1}(X\overline{(Y,Z)};\mathbb{Z})$ is free and its rank is $1$.

$\partial E(Y,Z)$ and $E(Y,Z)$ are connected by the assumption on the dimensions. 

$\partial E(Y,Z)$ is obtained by gluing $NE(Y,Z)$ and $SE(N,Z)$ along the boundaries $SNE(Y,Z)=\partial NE(Y,Z)$ and $\partial SE(N,Z)$.
$NE(Y,Z)$ is the total space of the trivial linear bundle over $\partial N(Z)$, diffeomorphic to the ($\dim Y-1$)-dimensional unit sphere $S^{\dim Y-1}$.
 $SE(Y,Z)$ is the total space of the trivial linear bundle over $CE(Z)$, diffeomorphic to a manifold obtained by removing the interior of a copy of the ($\dim Y$)-dimensional unit disk $D^{\dim Y}$ smoothly embedded in $Y$, whose fiber is the ($\dim X-\dim Y-1$)-dimensional unit sphere $S^{\dim X-\dim Y-1}$. 
$CE(Z)$ is also connected from Proposition \ref{prop:5} (\ref{prop:5.0}).
Due to the structures of these trivial linear bundles, the natural identifications there, and Seifert van-Kampen theorem, ${\pi}_1(\partial E(Y,Z))$ is isomorphic to ${\pi}_1(CE(Z))$.

${\pi}_1(\partial E(Y,Z))$ is mapped onto ${\pi}_1(E(Y,Z))$ by the isomorphism induced from the inclusion canonically. This is due to a nice structure of $E(Y,Z)$, having the structure of a trivial linear bundle over $CE(Z)$ whose fiber is the unit sphere $D^{\dim X-\dim Y}$. In addition, here $E(Y,Z)$ collapses to $CE(Z)$, which can be regarded as the image of some section of the bundle $E(Y,Z)$ over $CE(Z)$.

$X$ is a copy of the ($\dim X$)-dimensional unit disk $D^{\dim X}$ and $E(Y) \supset E(Y,Z) \supset CE(Z)$ is a closed tubular neighborhood of $Y$. Here $CE(Z)$ is regarded as the image of some section of the trivial bundle $E(Y,Z)$ over $CE(Z)$ as discussed just before. This bundle is the restriction of the trivial bundle $E(Y)$ over $Y$ to $CE(Z)$.

From this argument, we can see that
${\pi}_1(\partial E(Y,Z))$ is mapped onto  ${\pi}_1(E{(Y,Z)})$ by an isomorphism. This is also induced by the homomorphism induced from the inclusion canonically.

 $X$ is simply-connected.
 $X$ is regarded as a disk obtained by gluing two manifolds $E(Y,Z)$ and $X\overline{(Y,Z)}$ along the boundary $\partial E(Y,Z)$ and a connected component of the boundary $\partial X\overline{(Y,Z)}$. 
 
   By virtue of Seifert van-Kampen theorem. ${\pi}_1(X\overline{(Y,Z)};\mathbb{Z})$ is shown to be the trivial group.
   Thus $X\overline{(Y,Z)}$ is simply-connected. 
 
This completes the proof.
\end{proof}
\begin{proof}[A proof of Main Theorem \ref{mthm:1}]
By virtue of \cite{wall}, every $3$-dimensional closed, connected and orientable manifold can be embedded smoothly in ${\mathbb{R}}^5$. In Propositions \ref{prop:5} and \ref{prop:6}, let $\dim X=5$.

 For any finite cyclic group $G^{\prime}$, we can choose $Y \subset X$ as a $3$-dimensional closed, connected and orientable manifold and a one-point set $Z \subset Y$ enjoying the property
 $H_1(Y;\mathbb{Z}) \cong H_1(CE(Z);\mathbb{Z}) \cong G^{\prime}$ and $H_2(Y;\mathbb{Z}) \cong H_2(CE(Z);\mathbb{Z}) \cong \{0\}$.
For example, as $Y$, we can choose the $3$-dimensional unit sphere $S^3$ or a so-called {\it Lens space}. For example, ${\pi}_1(Y) \cong H_1(Y;\mathbb{Z}) \cong G^{\prime}$ holds for such a manifold.

For any finitely generated commutative group $G$, we can choose $Y$ as a manifold represented as a connected sum of manifolds diffeomorphic to some manifolds in these manifolds or $S^2 \times S^1$ for example.  

  Note that $CE(Z)$ is a manifold obtained by removing the interior of a smoothly embedded copy of the $3$-dimensional unit disk $D^3$ from $Y$. We also have
 $\dim X-\dim Y-1=5-3-1=1$ and $\dim X-1=5-1=4$.
 
   We consider a boundary connected sum of $X \overline{(Y,Z)}$ with $G$ being an arbitrary finite commutative group and finitely many copies of $5$-dimensional manifolds diffeomorphic to $S^2 \times D^3$, $S^3 \times D^2$ or $S^4 \times D^1$ to have a desired manifold $X$. We have our desired result by applying Propositions \ref{prop:5} and \ref{prop:6}.
Note that in the case $G_2=\{0\}$, we must use a manifold obtained by a $Y/Z$-remove to a $5$-dimensional standard sphere instead of a copy of the $5$-dimensional unit disk $D^5$. Note also that in the case where the manifold to which we do a $Y/Z$-remove is a $5$-dimensional standard sphere, we can argue similarly except for the $4$-th homology groups.

This completes the proof.
\end{proof}
We can show a similar theorem where the manifold $X$ is not supposed to be simply-connected. Proposition \ref{prop:4} implies that in the case where $H_1(X;\mathbb{Z})$ is the trivial group, $H_j(X;\mathbb{Z})$ must be free for $j=3,4$.
From Propositions \ref{prop:2} and \ref{prop:3} we immediately have Main Theorem \ref{mthm:2} as a corollary. This result can be regarded as an extension of a result (Theorem 8) in \cite{kitazawa4}, where the group $G$ is assumed to be represented as the direct sum of finitely many copies of $\mathbb{Z}$ or $\mathbb{Z}/2\mathbb{Z}$, which is a finite group of order $2$.
We can also improve Main Theorems 3 and 4 of \cite{kitazawa4} by applying the fact that there exists a $5$-dimensional closed and simply-connected manifold whose second integral homology group is isomorphic to a group represented as the direct sum of two copies of an arbitrary finite commutative group and which we can smoothly immerse and embed into ${\mathbb{R}}^6$. See also \cite{barden, nishioka} for such manifolds for example. Rigorous arguments are left to readers. 
\begin{Rem}
\label{rem:1}
The {\it Reeb space} of a (continuous) map between topological spaces is defined as the space of all connected components of preimages for the map. \cite{reeb} is one of pioneering papers on Reeb spaces.
For a fold map, it is regarded as a polyhedron whose dimension is equal to that of the target and it inherits (co)homological information of the manifold in various situations. $W_f$ is seen as the Reeb space of a special generic $f$.
If we see the compact manifold above as a Reeb space, then this operation is regarded as an advanced version of a kind of surgery operations first defined in \cite{kitazawa} where $Z$ is empty under very specific and explicit scenes. Note also that operations in \cite{kitazawa} were introduced motivated by \cite{kobayashi,kobayashi2,kobayashisaeki}, which construct explicit fold maps satisfying some conditions systematically and that they are extensions of operations in \cite{kobayashi, kobayashi2}.

\end{Rem}
Related to Remark \ref{rem:1} and all our present study, constructing special generic maps and fold maps is natural, and fundamental. It is also difficult, even for manifolds of very explicit classes consisting of manifolds which are elementary in some meaningful senses. We will discover more on construction of explicit special generic maps and fold maps. 
\section{Acknowledgement, grants and data.}
The author is a member of JSPS KAKENHI Grant Number JP17H06128 "Innovative research of geometric topology and singularities of differentiable mappings" (Principal Investigator: Osamu Saeki). This work is also supported by this project. The present paper contains all associated data essentially supporting this study.

\end{document}